\documentclass[preprint,12pt]{elsarticle}

\usepackage{amssymb}
\usepackage{amsmath}
\usepackage{geometry}   
\usepackage[colorlinks,citecolor = red, linkcolor=blue,hyperindex]{hyperref}
\usepackage{euscript,eufrak,verbatim, mathrsfs}
\usepackage{bbm}
\usepackage{graphicx}
\usepackage{float}
\usepackage{float, tikz}
\usepackage{caption}

\usepackage{extpfeil}

\usepackage[all, cmtip]{xy}

\usepackage{upref, xcolor, dsfont}
\usepackage{amsfonts,amstext,amsbsy, amsopn,amsthm}
\usepackage{enumerate}

\usepackage{url}

\usepackage{mathtools}
\usepackage{bookmark}

\usepackage{euscript}
\usepackage{helvet}         
\usepackage{courier}        
\usepackage{type1cm}        
\usepackage{multicol}        
\usepackage[bottom]{footmisc}

\newtheorem{theorem}{Theorem}[section]
\newtheorem*{theorem*}{Theorem B}
\newtheorem{lemma}[theorem]{Lemma}

\newtheorem{proposition}[theorem]{Proposition}
\newtheorem{corollary}[theorem]{Corollary}

\newtheorem*{definition*}{Definition}
\newtheorem*{remark*}{Remark}

\newtheorem*{observation*}{Observation}

\newtheorem*{assumption*}{Assumption}
\newtheorem*{question*}{Question}
\newtheorem*{problem*}{Problem}

\newtheorem*{example*}{Example}

\newcommand{\R}{\mathbb{R}}
\newcommand{\N}{\mathbb{N}}

\newcommand{\E}{\mathbb{E}}

\newcommand{\Tr}{\mathrm{Tr}}

\newcommand{\Var}{\mathrm{Var}}

\newcommand{\an}{\text{\, and \,}}

\newcommand{\indi}{\mathds{1}}

\journal{Stochastic Processes and their Applications}

\begin{document}

\begin{frontmatter}



\title{Spectral measure of large random Helson matrices}

\author[label1]{Yanqi Qiu}
\ead{yanqiqiu@ucas.ac.cn}

\author[label2]{Guocheng Zhen\corref{cor1}}
\ead{guochengzhen@whu.edu.cn}

\affiliation[label1]{organization={School of Fundamental Physics and Mathematical Sciences, HIAS, University of Chinese Academy of Sciences}, 
	city={Hangzhou},
	postcode={310024}, 
	state={Zhejiang},
	country={China}}

\affiliation[label2]{organization={School of Mathematics and Statistics, Wuhan University}, 
	city={Wuhan},
	postcode={430072}, 
	state={Hubei}, 
	country={China}}

\cortext[cor1]{Corresponding author}

\begin{abstract}
We study the limiting spectral measure of large random Helson matrices and large random matrices of certain patterned structures.

Given a real random variable $X \in L^{2+ \varepsilon}(\mathbb{P}) $ for some $\varepsilon > 0$ and $\mathrm{Var}(X) = 1$. For the random $n \times n$ Helson matrices generated by the independent copies of $X$, scaling the eigenvalues by $\sqrt{n}$, we prove the almost sure weak convergence of the spectral measure to the standard Wigner semi-circular law. Similar results are established for  large random matrices with certain general patterned structures.
\end{abstract}


\end{frontmatter}

\section{Introduction}\label{introduction}
The Helson matrix, also known as the multiplicative Hankel matrix, is a fundamental object in the theory of Hardy spaces of Dirichlet series introduced by Helson in \cite{Henry-Helson}. It plays a similar role as (additive) Hankel matrices play in the analysis of holomorphic functions in the unit disk. As such, questions regarding whether or not classical results for Hankel matrices can be extended to the multiplicative setting have attracted considerable attentions. Significant research has been conducted on this matrix, including studies on its explicit eigenvalue asymptotics \cite{Helson}, various operator properties \cite{finite} and many other works (see e.g. \cite{P-onto, fo, lower-bound}).

Large-dimensional random matrix theory (LDRM) originated in the early 1940s as a tool for analyzing the complex structure of heavy nuclei within quantum mechanics. This pioneering application led to the introduction of the well-known Wigner matrix and this marked the beginning of LDRM as a distinct and rapidly evolving field of mathematical research.

One of the main focus in the field of LDRM is the study of limiting distribution of the eigenvalues and   a remarkable  achievement  in  Wigner's work is the establishement of the semi-circular law for the limiting distribution of its eigenvalues \cite{Wigner}.

Following the work of Wigner, LDRM saw significant theoretical expansion and  many fundamental results were established on the limiting spectral distributions of various random matrix models. These include the Marchenko-Pastur law for sample covariance matrices, the limiting spectral measures of Hankel and Toeplitz random matrices \cite{Hankel}, and the identification of the symmetrized Rayleigh law and the Gaussian law \cite[Sections 8.2, 8.6-8.7]{patterned}. Throughout this development, Gaussian Ensembles---the random Hermitian or real symmetric matrices with Gaussian entries---have held a central position. Research on the Gaussian Unitary Ensemble (GUE), in particular, is highly advanced, featuring profound theoretical results \cite[Sections 2-3]{Intro-RM}.

 LDRM models featuring specific patterns arise naturally across diverse disciplines, including econometrics, computer science, mathematics, physics, and statistics. Consequently, the existence and properties of the limiting spectral distribution for patterned random matrices—such as Reverse Circulant and Elliptic matrices---as the dimension grows have been extensively studied \cite[Sections 8-10]{patterned}. Recent research continues to explore advanced topics within the field, such as the asymptotic freeness of Gaussian random matrices \cite{MC}.

\subsection{Statement of the main results}
In this paper, we give the results on the limiting spectral measure of large random Helson matrices as well as that of large random patterned matrices with specific structures.

We start with recalling the definition of the spectral measure. For a symmetric $n \times n$ real matrix $A$, let $\lambda_{j}(A), 1 \leq j \leq n$, denote the eigenvalues of the matrix $A$, written in a nonincreasing order. The spectral measure of $A$, denoted $\widehat{\mu}(A)$, is given by
\[
\widehat{\mu}(A) = \frac{1}{n} \sum_{i=1}^{n} \delta_{\lambda_{j}(A)}.
\]

Now, let $\N=\{1, 2, \cdots\}$ be the semigroup of positive integers and let $mul$ denote the usual multiplication map $mul: \N\times \N\rightarrow\N$ given by $mul(i,j)= ij$. Let $X$ be a real random variable. For $n \in \mathbb{N}$, we define the corresponding random $n \times n$ Helson matrix as follows:
let $\{X_{k}:k=1,2,...\}$ be the independent copies of $X$, set $H_{n} = [X_{mul(i,j)}]_{1 \leq i,j \leq n}$, with $mul(i,j) = ij$. More precisely,
\[ 
H_{n} = \begin{bmatrix}
	X_{1} & X_{2} &  X_3 & \cdots & X_{n} \\
	X_{2} & X_{4} &  X_6&  \cdots & X_{2n} \\
	X_3&X_6 & X_9 & \cdots & X_{3n}\\
	\vdots & \vdots & \vdots &\ddots & \vdots \\
	X_{n} & X_{2n} & X_{3n}& \cdots & X_{n^{2}} 
\end{bmatrix}.
\]

Recall that the standard Wigner semi-circular law $\gamma_{sc}$ is the probability measure supported on [-2,2] given by
\begin{equation}\label{semi-circle}
	d\gamma_{sc}(x) = \mathds{1}_{[-2,2]} \frac{\sqrt{4-x^2}}{2 \pi}dx.
\end{equation}

\begin{theorem}\label{thm-main}
	Let $X$ be an arbitrary real random variable such that  $\Var(X)=1$ and  $\E[|X|^{2+\varepsilon}]<\infty$ for some $\varepsilon >0$.  Let $H_{n}$ be the corresponding random $n \times n$ Helson matrix. Then, with probability 1, $\widehat{\mu}(H_{n} / \sqrt{n})$ converges weakly to the Wigner semi-circular law $\gamma_{sc}$ as $n \rightarrow \infty$.
\end{theorem}

\begin{remark*}
	Simulations suggest that Theorem~\ref{thm-main} is likely to be true, see Figure~\ref{Figure 1} where we compare the histograms of the empirical distribution of eigenvalues of 100 realizations of the random $1000 \times 1000$ Helson matrices generated by standard Gaussian variables with the Wigner semi-circular law $\gamma_{sc}$.
\end{remark*}

\begin{figure}[htbp]
	\centering
	\includegraphics[width=0.6\textwidth]{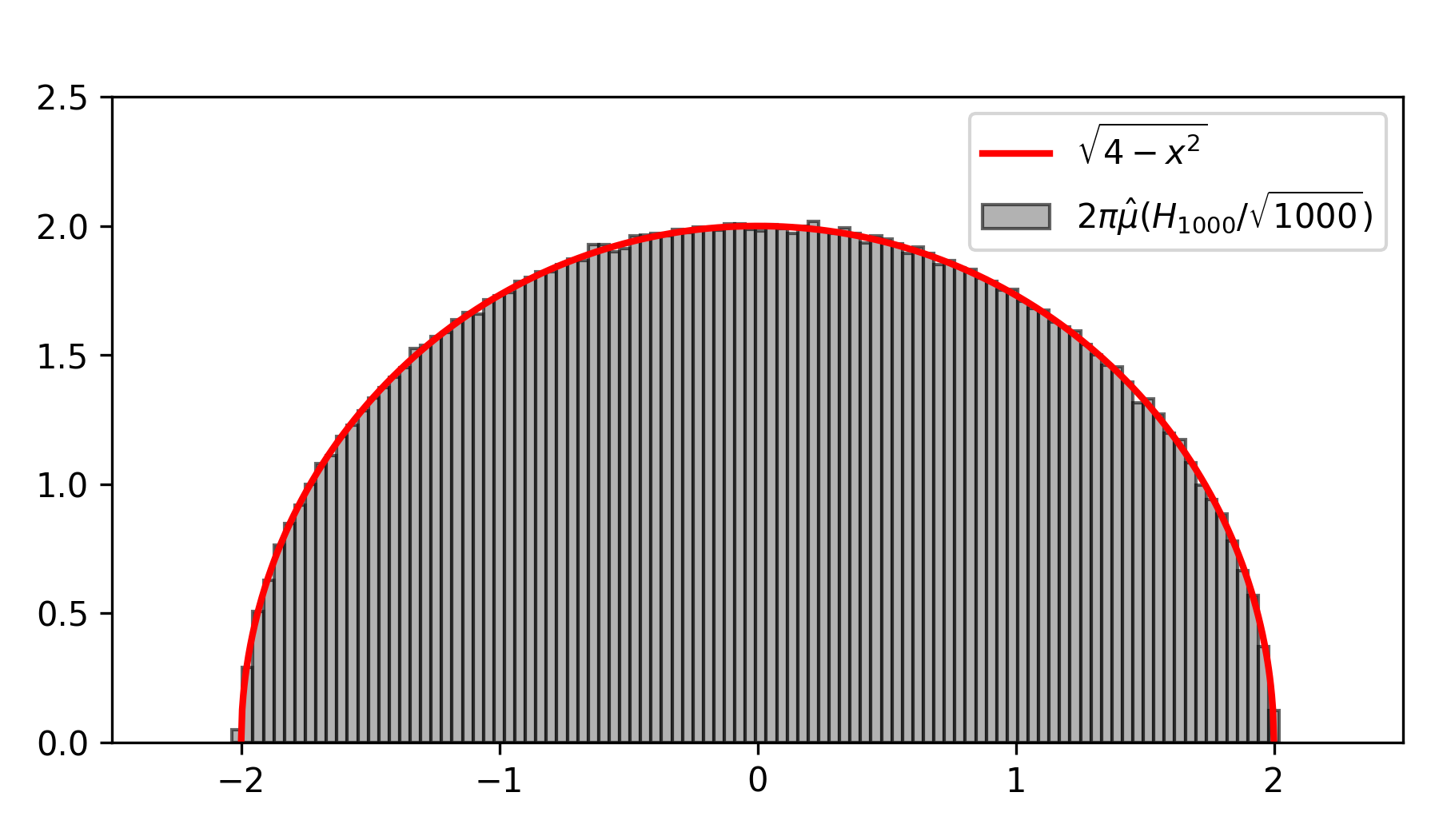}
	\caption{Comparison between $\widehat{\mu}(H_{1000}/\sqrt{1000})$ and the Wigner semi-circular law $\gamma_{sc}$. (generated by standard Gaussian variables and both scaled up by $2\pi$ times)}
	\label{Figure 1}
\end{figure}

Theorem~\ref{thm-main}  can be generalized to more abstract setting. For presenting this result, we need to introduce more notations.

Let $I$ be an index set and $S \colon \mathbb{N}^2 \to I$ be a map. For any $i \in I$, let
\[
S_i := \{(x,y) \in \mathbb{N} \times \mathbb{N} | S(x,y) = i\} 
\] 
be the level set of the map $S$. Let $X$ be a real random variable. For $n \in \mathbb{N}$, we define the corresponding random $n \times n$ $S$-patterned matrix $H_{n}^{S}$ as follows:
Let $\{X_{i}\}_{i \in I}$ be the independent copies of $X$, set
\[
H_{n}^{S}=[X_{S(i,j)}]_{1 \leq i,j \leq n}.
\]
As usual, let $[n]=\{1,2,\cdots,n\}$. We denote the truncation of the level set $S_{i}$ onto the square $[n]^2$ by 
\[
S_i^{(n)} = S_i \cap [n]^2
\]
and denote the number of the elements in $S_i^{(n)}$ by $|S_i^{(n)}|$. 

Our key assumption on the map $S$ is that it satisfies the following $C$-condition:
\begin{enumerate}
	\item[(C1)] Symmetry: $S(x,y) = S(y,x)$ for any $x,y \in \mathbb{N}$;
	\item[(C2)] Coordinatewise injectivity: Fix any $ y \in \mathbb{N}$, the map $S(\cdot,y)$ is injective;
	\item[(C3)] Small dimension assumption: 
	\begin{equation}\label{sda}
		\sum_{i \in I} |S_i^{(n)}|^2 = o(n^3) \,\, \text{as $n\to\infty$}.
	\end{equation}
\end{enumerate}

\begin{remark*}
	The small dimension assumption $(\ref{sda})$ is equivalent to
	\[
	\lim_{n \to \infty} \frac{\#\{(x,y,z,w) \in [n]^4| S(x,y)=S(z,w) \}}{n^3} = 0. 
	\]
\end{remark*}

\begin{theorem}\label{thm-gen-semi}
	Assume that $S$ is a map satisfying the $C$-condition. Let $X$ be an arbitrary bounded real random variable such that  $\Var(X)=1$. Let $H_{n}^{S}$ be the corresponding random $n \times n$ $S$-patterned matrix. Then with probability 1, $\widehat{\mu}(H_{n}^{S} / \sqrt{n})$ converges weakly to the Wigner semi-circular law $\gamma_{sc}$ as $n \rightarrow \infty$.
\end{theorem}

\subsection{Outline of the Proofs of Theorem~\ref{thm-main} and Theorem~\ref{thm-gen-semi}}
\label{subsec:outline}

We now give the outline of the proofs of Theorem~\ref{thm-main} and Theorem~\ref{thm-gen-semi}:

\begin{itemize}
    \item {\bf Step 1. Reduction to bounded entries (Proposition~\ref{prop-TB}).}
    
    By a standard truncation and rescaling argument (cf.~\cite{Hankel}), carried out in Proposition~\ref{prop-TB}, we may assume without loss of generality that the entries $X_j$ are centered and uniformly bounded. 

  \item {\bf Step 2. Moment computations for bounded entries (Proposition~\ref{moments}).}
    \begin{enumerate}
      \item To describe the moments of the average spectral measure of random $S$-patterned matrix in Theorem~\ref{thm-gen-semi},  for an integer $r$, we need to estimate the following quantity
      \begin{align*}
		\begin{split}
			\mathbb{E}[\Tr((H_n^{S})^{r})] 
			&= \sum_{i_1,\cdots,i_{r}=1}^{n} \mathbb{E}[H_{n}^{S}(i_1,i_2)H_{n}^{S}(i_2,i_3) \cdots H_{n}^{S}(i_r,i_1)]\\
			&=\sum_{i_1,\cdots,i_{r}=1}^{n} \mathbb{E}[X_{S(i_1,i_2)} X_{S(i_2,i_3)} \cdots X_{S(i_r,i_1)}].\\
		\end{split}
	\end{align*}

    \item Following the methods used in \cite{Hankel}, we introduce the combinatorics object circuits and get the following equation:
    \[
    \mathbb{E}[\Tr((H_n^{S})^{r})] = \sum_{\pi} \mathbb{E}[X_{\pi}],
    \]
    where the sum is over all circuits in $\{1, \ldots, n\}$ of length $r$ (see the beginning of \S \ref{sec-cir} for the precise definition of circuits and $S$-matched ones). In order to divide the summation terms into main contribution terms and negligible terms: 
    \begin{align}\label{cir-dec}
     \mathbb{E}[\Tr((H_n^{S})^{r})]  =  \underbrace{\sum_{\text{pair-matched circuits $\pi$}} \mathbb{E}[X_{\pi}]}_{T_{main}(n, r)} + \underbrace{\sum_{\text{other circuits $\pi$}} \mathbb{E}[X_{\pi}]}_{T_{neg}(n, r)},
    \end{align}
    we introduce the system of equations \eqref{equation} associated to the map $S$ which satisfied the $C$-condition in Theorem~\ref{thm-gen-semi} as the authors did in \cite{Hankel}.   In particular, the $C$-condition is used to show that the negligible terms $T_{neg}$ is negligible in the sense that,  for an integer $r\ge 1$, we have 
    \begin{align}\label{neg-is-neg}
    \lim_{n\to\infty} \frac{1}{n^{\frac{r}{2}+1}} T_{neg}(n, r) = 0. 
    \end{align}

      \item Then we use some key lemmas about the number of $S$-matched circuits in $\{1,\cdots,n\}$ of length $r$ with at least one edge of order $3$ (see Lemma~\ref{lem-3ord}) given in \cite{Hankel}, together with properties about partition words and Catalan words (Lemma~\ref{Catalan}, Lemma~\ref{deleting}, Corollary~\ref{non-cartalan word}, Lemma~\ref{p_n(w)}, Lemma~\ref{noncartalan p_n(w)}) to calculate the main contribution terms and get the limit moments of the average spectral measure: for $k \in \mathbb{N}$,
      \[
      \lim_{n \to \infty} \frac{1}{n^{k+1}} \mathbb{E}\left[\Tr((H_n^{S})^{2k})\right] = \frac{(2k)!}{(k+1)!k!},
      \]
      and
      \[
		\lim_{n \to \infty} \frac{1}{n^{k+1/2}} \mathbb{E}\left[\Tr((H_n^{S})^{2k-1})\right] = 0.
      \]
    \end{enumerate}

  \item {\bf Step 3. Concentration of moments of the spectral measure (Proposition~\ref{4-order trace}).}
  
     Using Lemma~\ref{quadruples} to set up a suitable $L^{4}$ bounds on trace moments, and then applying Chebyshev's inequality, the Borel-Cantelli lemma and the method of moments, we deduce the almost-sure weak convergence 
     \[
     \widehat{\mu}(H_{n}^{S}/\sqrt{n}) \xrightarrow[n \to \infty]{w} \gamma_{sc},
     \]
     completing the proof of Theorem~\ref{thm-gen-semi}.

  \item {\bf Step 4. Specializing to random Helson matrices (Proposition~\ref{multi}).}
   By verifying that the usual multiplication map $mul(i, j )= ij$ is a special case of the map $S$ satisfying the $C$-condition in Theorem~\ref{thm-gen-semi}, we derive Theorem~\ref{thm-main} from Theorem~\ref{thm-gen-semi}.
\end{itemize}

\subsection{Discussion on some assumptions}

In this subsection, we discuss about some assumptions we made in Theorem~\ref{thm-main} and Theorem~\ref{thm-gen-semi}.
\begin{itemize}
    \item {\bf 1. The intuition of small dimension condition (C3).}

    When we set up the models (random Helson matrices $H_{n}$ as well as $S$-patterned matrices $H_{n}^{S}$) in Theorem~\ref{thm-main} and Theorem~\ref{thm-gen-semi}, these patterns lead to some equal entries in the random matrices $H_n$ and $H_n^S$. In the random matrices $H_n$ and $H_n^S$,  the entries are more dependent than the situation of  the classical Wigner matrices. The intuition of small dimension assumption (C3):
    \[
	\lim_{n \to \infty} \frac{\#\{(x,y,z,w) \in [n]^4| S(x,y)=S(z,w) \}}{n^3} = 0
    \]
    is that even though some equal entries may lead to a lack of independence, this assumption guarantees not too much dependence between the entries.

  More precisely,  the small dimension assumption will be  used  in the proof of \eqref{neg-is-neg} (which is a consequence of  Lemma~\ref{noncartalan p_n(w)} below).  In particular, this small dimension assumption  is a sufficient condition to  prove Lemma~\ref{noncartalan p_n(w)}. Therefore,  we may weaken this condition if we can estimate directly the number of solutions of the equation system \eqref{equation} associated to a non-Catalan word. See the proof of Lemma~\ref{noncartalan p_n(w)} for details.

     \item {\bf 2. The bounded condition in Theorem~\ref{thm-gen-semi}.}

 In Theorem~\ref{thm-gen-semi},    we require that the variable $X$ is bounded,   while in  Theorem~\ref{thm-main}, we require $\mathbb{E}[|X|^{2 + \varepsilon}] < \infty$ for some $\varepsilon > 0$. The main reason is that we establish two number theory results \eqref{d_n} and \eqref{A_n} in \S \ref{number-theo1} and \S \ref{number-theo2}. Hence we can take advantage of the truncation technique and the sufficient conditions for almost sure convergence in double arrays (see Lemma~\ref{lem-Teicher}) to get a crucial limit in the proof of Proposition~\ref{prop-TB}.
      Therefore, if we can establish similar estimations for the abstract map $S$, then we may weaken the boundedness condition in Theorem~\ref{thm-gen-semi}.

     \item {\bf 3. The condition $\mathbb{E}[|X|^{2 + \varepsilon}] < \infty$ for some $\varepsilon > 0$ in Theorem~\ref{thm-main}.}

    In the assumption of Theorem~\ref{thm-main}, we need the condition that $\mathbb{E}[|X|^{2+\varepsilon}]<\infty$ for some $\varepsilon > 0$ rather than $\mathbb{E}[|X|^{2}]<\infty$. In fact, this is a technical assumption used in the proof of Proposition~\ref{prop-TB}, in which a bound on the number of different values in the $n \times n$ multiplication table is needed.

To compare the conditions $\E[|X|^{2}]<\infty$ with the condition $\mathbb{E}[|X|^{2 + \varepsilon}] < \infty$ for some $\varepsilon > 0$, we simulate two numerical experiments (see Figure~\ref{fig-comp}): we consider the probability density function:
    \[
    f_{X}^{\epsilon_0}(x) \propto \frac{1}{x^{3+\epsilon_{0}}(\log x)^{2}} \mathds{1}_{x \geq e},
    \]
    where $\epsilon_0 \geq 0$. Hence the variable $X$ has only a $(2+\epsilon_0)$ moment and no higher moments. In particular, if we choose $\epsilon_0 = 0$, then $X$ has only a second moment and no higher moments. 
   \begin{figure}[htbp]
      \centering
      \includegraphics[width=0.49\textwidth,keepaspectratio]{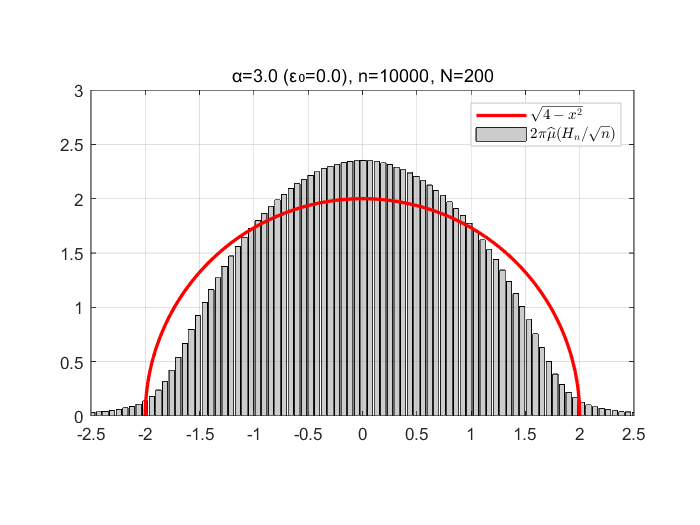}%
      \hfill
      \includegraphics[width=0.49\textwidth,keepaspectratio]{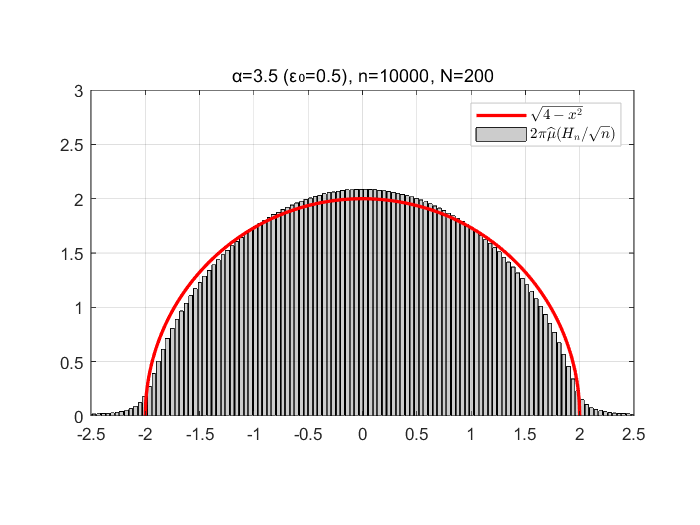}
      \caption{
        Comparison between $\epsilon_0 = 0$ and $\epsilon_0 = 0.5$ with matrix size $n = 10000$ and the number of realizations $N = 200$.
      }\label{fig-comp}
    \end{figure}
Compared with Figure~\ref{Figure 1}, we can observe that the convergence rate generated by the standard Gaussian variable is better than generated by the variable which has the probability density function $f_{X}^{0.5}(x)$. From Figure \ref{fig-comp}, we do not know  whether the spectral measure for $\epsilon_0 = 0$ weakly converges to the Wigner semi-circular law $\gamma_{sc}$ or not. A possible reason is the two very different behaviors of the tails of these two distributions (which decay super exponentially for the Gaussian one and polynomially for the other one).


\end{itemize}

\section{Preliminaries}
We will use the notation $\asymp$ as follows:
\[
f \asymp g \Leftrightarrow \exists C_{1}, C_{2} > 0: C_{1}|g| \leq f \leq C_{2}|g|.
\]
\subsection{The bounded lipschitz metric}
Let $d_{BL}$ denote the bounded Lipschitz metric
\begin{equation*}
	d_{BL}(\mu, \nu) = \sup \left\{ \int f \,d\mu - \int f \,d\nu : \|f\|_{\infty} + \|f\|_{L} \leq 1 \right\},
\end{equation*}
where $\|f\|_{\infty} = \sup_{x}|f(x)|$, $\|f\|_{L} = \sup_{x\neq y}|f(x) - f(y)|/|x - y|$. It is well known (cf. \cite[Section 11.3]{BL}) that $d_{BL}$ is a metric for the weak convergence of probability measures. Note that, for the spectral measures of $n \times n$ symmetric real matrices $A, B$ we have
\begin{align*}
	d_{BL}(\widehat{\mu}(A), \widehat{\mu}(B)) &\leq \sup \left\{ \frac{1}{n} \sum_{j=1}^{n} |f(\lambda_j(A)) - f(\lambda_j(B))| : \|f\|_{L} \leq 1 \right\} \nonumber \\
	&\leq \frac{1}{n} \sum_{j=1}^{n} |\lambda_j(A) - \lambda_j(B)|.
\end{align*}
Moreover, by Lidskii's theorem \cite{Lidski}
\begin{equation*}
	\sum_{j=1}^{n} |\lambda_j(A) - \lambda_j(B)|^2 \leq \Tr((B - A)^2),
\end{equation*}
we have
\begin{equation}\label{lidskii}
	d_{BL}^2(\widehat{\mu}(A), \widehat{\mu}(B)) \leq \frac{1}{n} \Tr((B - A)^2).
\end{equation}

\subsection{A useful rank inequality}
For an $n \times n$ Hermitian matrix $A$, let $\lambda_{j}(A), 1 \leq j \leq n$, denote the eigenvalues of the matrix $A$, written in a nonincreasing order, let 
\[
F^{A}(x) = \frac{\#\left\{k \leq n, \lambda_{k}(A) \leq x \right\}}{n}
\]
be the distribution function of the spectral measure $\widehat{\mu}(A)$.

The following result will be useful for us \cite[Lemma 2.2]{RI}:
let $A$ and $B$ be two $n \times n$ Hermitian matrices, then
\begin{equation}\label{rank inequality}
	\vert\vert F^{A} - F^{B} \vert\vert_{\infty} \leq \frac{1}{n} \mathrm{rank}(A-B).
\end{equation}

\subsection{Asymptotic of the number of divisors}\label{number-theo1}
For analyzing the spectral measure of random Helson matrices, we will need the following results from number theory. For any integer $n\ge 1$, denote 
\[
D(n)=\{l \in \mathbb{N} | \text{$l$ is a divisor of $n$}\}
\] 
the set of divisors of $n$ and let $d(n) = \#D(n)$ be the number of the divisors of $n$. By convention, we set  $d(x) = 0$ if $x$ is not an integer. Denote
\[
m(k;n) =   \#\{(i,j) | i,j \in [n],ij=k\}
\]
the number of times $k$ appears in the $n \times n$ multiplication table.
It is well known that (cf. \cite[Theorem 13.12]{divisor}) for all $\varepsilon > 0$,
\begin{equation}\label{d_n}
	d(n) = o(n^\varepsilon)  \, \, \text{as $n\to\infty$}.
\end{equation}
More precisely,
\begin{equation*}
	\limsup_{n \to \infty} \frac{\log d(n)}{\log n / \log \log n} = \log 2 < 1.
\end{equation*}
It follows that 
\begin{equation}\label{precise}
	d(n) = O(n^\frac{1}{\log \log n})  \,\, \text{as $n\to\infty$}.
\end{equation}

\subsection{Numbers of entries in multiplication table}\label{number-theo2}
For any integer $n\ge 1$, let
\[
A_n = \{xy|x,y \in [n]\}
\]
denote the set of distinct entries in the $n \times n$ multiplication table. Ford \cite[Corollary 3]{Ford} showed that the cardinality $|A_n|$ of the set $A_{n}$ has the following precise asymptotic order as $n \to \infty$:
\begin{equation}\label{A_n}
	|A_n| \asymp \frac{n^2}{(\log n)^c (\log \log n)^{\frac{3}{2}}},
\end{equation}
where 
\begin{equation*}
	c = 1 - \frac{1+\log\log2}{\log2} = 0.086071\ldots > 0.
\end{equation*}

\subsection{Moment methods}
For a Borel probability measure $\gamma$ on $\mathbb{R}$, we denote its $k$-th moment by
\[
m_{k}(\gamma) = \int_\R x^{k} \gamma(dx).
\]
Since $\gamma_{sc}$ has compact support, by Stone - Weierstrass theorem, it is uniquely determined by its moments: 
\[
m_n(\gamma_{sc}) =
\begin{cases}
	0 & \text{if } n=2k-1, \\
	C_{k}=\frac{1}{k+1}\binom{2k}{k} & \text{if } n=2k.
\end{cases}
\]

Let \( \widetilde{\mathcal{P}}(\mathbb{R}) \) be the set of probability measures $\mu$ on $\mathbb{R}$ such that $\int_{\mathbb{R}} \vert x^k \vert d\mu < \infty$ for all $k \in \mathbb{N}$. The following classical lemma (cf. \cite[Theorem 3.5]{method}) will be useful for us.

\begin{lemma}[From moments convergence to weak convergence]\label{m to w}
	Assume that \( \mu, \mu_1, \mu_2, \ldots \) all belong to \( \widetilde{\mathcal{P}}(\mathbb{R}) \) and
	\[ \lim_{n \to \infty} m_{k}(\mu_n) = m_{k}(\mu) \,\, \text{for all $k=0,1,\cdots$}. \]
	Assume moreover that \( \mu \) is characterized by its moments, then $\mu_n$ converges weakly to $\mu$.
\end{lemma}	

\section{Some combinatorics on partition words and circuits} \label{comb-property}
The partition words and the Catalan words will be useful for us in describing the structure of the moments of $\gamma_{sc}$. Let us introduce some properties of them.
\subsection{The properties of partition words and Catalan words}
Consider the set of letters $\mathbb{N}=\{1,2,\cdots\}$ equipped with the usual order. Thus, by a word, we mean a finite string of integer numbers from $\N$: 
\[
1234,    3321134,  4356,  etc.
\]  
In what follows, we will use the English letters $a,b,c,d$  and so on  to denote arbitrary letters from $\N$ (therefore,  there is no a priori order between the English letters $a, b, c, d$).  

By a pair partition, we mean a partition of a set into its subsets of cardinality 2. It is convenient to index the pair partitions by the partition words $w$; these are words of length $2k$ with $k$ pairs of letters such that the first occurrences of each of the $k$ letters are in alphabetic order. 

Equivalence classes may be identified with partitions of $\{1,2,\cdots,2k\}$, we will say two partition words $w_1,w_2$ are equivalent, and write $w_1 \sim w_2$, if they induce the same partition. In what follows, we will not distinguish the word and its equivalence class.

For example, in the case $k = 2$ we have $3$ such partition words

\begin{center}
	$1122, \quad 1221, \quad 1212,$
\end{center}
which correspond to the pair partitions

\begin{center}
	$\{1, 2\} \cup \{3, 4\} \quad \{1, 4\} \cup \{2, 3\} \quad \{1, 3\} \cup \{2, 4\}.$
\end{center}
while
\begin{center}
	$1112, \quad 1222, \quad 1111$
\end{center}
are not partition words.

\begin{definition*}
	A partition word $w$ of length $2k$ is called a Catalan word of length $2k$ if
	\begin{itemize}
		\item there is at least one double letter,
		\item sequentially deleting all double letters eventually leads to the empty word (which will be denoted by $\emptyset$).
	\end{itemize}
\end{definition*}

For example, $122133$, $12332441$ are Catalan words with the sequentially deleting procedure 
\[
122133 \to  11 \to \emptyset \an 12332441 \to 1221 \to 11 \to \emptyset. 
\]
while $121332$ is a non-Catalan word, since the sequentially deleting procedure is not $\emptyset$:
\begin{equation}\label{121332}
	121332 \to 1212 \ne \emptyset.
\end{equation}

The following result is well known \cite[Lemma 3]{Cata}.
\begin{lemma}\label{Catalan}
	The number of Catalan words of length 2k is given by the Catalan numbers 
	\[
	C_{k}=\frac{1}{k+1}\binom{2k}{k}=\frac{(2k)!}{(k+1)!k!}.
	\]
\end{lemma}

Given a partition word $w$, denote $\mathcal{A}(w)$ the set of all the letters used in $w$.

A word is called a reduced word, if it does not contain any double letter.  If $w'$ is the word after sequentially deleting all double letters of $w$, then $w'$ is a reduced word and  we will call it the reduced word obtained from $w$. For example, by $(\ref{121332})$, $1212$ is the reduced word obtained from $121332$.

The next lemma gives the structure of the reduced word.

\begin{lemma}[The structure of the reduced words]\label{deleting}
	Given a partition word $w$, denote $w'$ the reduced word obtained from $w$, then $w'$ must be of one of the following forms:
	\begin{itemize}
		\item $w' = \emptyset$,
		\item $w' = w_{1}abw_{2}a$, where $w_{1} \neq \emptyset$, $w_{2}$ is possibly an empty word.
		\item $w' = w_{1}abw_{2}acw_{3}$, where $w_{1} \neq \emptyset$, $w_{2}, w_{3}$ are possibly empty words.
	\end{itemize}
	where $a,b,c$ are three distinct letters and $b,c \in \mathcal{A}(w_1)$.
\end{lemma}

\begin{proof}
	In a word, we say a position is new if it is the first occurrence position of a letter.  Otherwise, we say the position is not new. 
	
	Let $w'$ be the reduced word obtained from the partition word $w$.  Assume that $w'\neq \emptyset$.  First we choose the last new position of the word $w'$ and denote the letter on this last new position by $a$.   Then the word $w'$ can be written as 
	\[
	w'= w_{1} a  \underbrace{\underbrace{\cdots}_{\text{non-empty}} a \underbrace{\cdots}_{\text{possibly empty}}}_{\text{no new positions}}\,\, \text{with $w_1\neq \emptyset$}.
	\] 
	Now let us concern about the adjacent letter on the right side of the first $a$, denote it as $b$. Since there is not any new position after the first $a$, there must be a $b$ in word $w_1$. Denote the word between the second $b$ and the second $a$ as $w_2$, so $w'$ can be written as 
	\[
	w' = w_{1}abw_{2}a\underbrace{\cdots}_{\text{possibly empty}},
	\]
	with $b \in \mathcal{A}(w_1)$ (the word $w_2$ is possibly an empty word). Now there are exactly two cases: 
	
	{\flushleft \it	Case 1:}
	the second $a$ is on the far right position. Then  
	\[
	w' = w_{1}abw_{2}a.
	\]
	{\flushleft\it		 Case 2:}
	the second $a$ is not on the far right position.  In this case, we denote the letter on the right side of the second $a$ as $c$. Similarly, $c \in \mathcal{A}(w_1)$, then 
	\[
	w' = w_{1}abw_{2}acw_{3}, 
	\]
	where the word $w_3$ is possibly an empty word. 
\end{proof}

\begin{corollary}[The structure of non-Catalan partition words]\label{non-cartalan word}
	Let $w$ be a non-Catalan partition word. Then $w$ must be of one of the following forms:
	\begin{itemize}
		\item $w= w_{1}aw_{2}bw_{3}aw_{4}$, where $a,b$ are two distinct letters, $b \in \mathcal{A}(w_1)$ and $w_{2}, w_{4}$ are Catalan words while $w_1$ is a non-Catalan word and $w_3$ is possibly a Catalan word.
		\item $w = w_{1}aw_{2}bw_{3}aw_{4}cw_{5}$, where $a,b,c$ are three distinct letters, $b,c \in \mathcal{A}(w_1)$ and $w_{2}, w_{4}$ are Catalan words while $w_1$ is a non-Catalan word and $w_3, w_5$ are possibly Catalan words.
	\end{itemize}
\end{corollary}

\begin{proof}
	Let $w$ be a non-Catalan partition word.   Let $w'$  be the reduced word obtained from $w$. Then $w'\neq \emptyset$. By  Lemma~\ref{deleting}, the reduced word $w'$ must be of one of the following forms:
	{\flushleft \it Case 1:}  there exist two distinct letters $a, b$ such that 
	\[
	w' = w_{1}'abw_{2}'a  \an  b \in \mathcal{A}(w_1'). 
	\]
	
	{\flushleft\it    Case 2:} there exist three distinct letters $a, b, c$ such that 
	\[
	w' = w_{1}'abw_{2}'acw_{3}'  \an  b, c\in \mathcal{A}(w_1'). 
	\]

	Therefore,  we can revert $w$ by adding the Catalan words (which are deleted from the reducing procedure) into all possible intervals of $w'$ and conclude that $w$ must be of one of the following forms:
	
	Case 1 corresponds to
	\[
	w= \hat{w}_{0}w_{1}'\hat{w}_{1}a\hat{w}_{2}b\hat{w}_{3}w_{2}'\hat{w}_{4}a\hat{w}_{5},
	\]
	where $\hat{w}_0, \hat{w}_1, \cdots, \hat{w}_5$ are Catalan words. Rewrite the adjacent words to one word and rename it, then we can get 
	\[
	w= w_{1}aw_{2}bw_{3}aw_{4},
	\]
	where $a,b$ are two distinct letters, $b \in \mathcal{A}(w_1)$ and $w_{2}, w_{4}$ are Catalan words while $w_1$ is a non-Catalan word and $w_3$ is possibly a Catalan word.
	
	Case 2 corresponds to
	\[
	w = \hat{w}_{0}w_{1}'\hat{w}_{1}a\hat{w}_{2}b\hat{w}_{3}w_{2}'\hat{w}_{4}a\hat{w}_{5}c\hat{w}_{6}w_{3}'\hat{w}_{7},
	\]
	where $\hat{w}_0, \hat{w}_1, \cdots, \hat{w}_7$ are Catalan words. Rewrite the adjacent words to one word and rename it, then we get 
	\[
	w = w_{1}aw_{2}bw_{3}aw_{4}cw_{5},
	\]
	where $a,b,c$ are three distinct letters, $b,c \in \mathcal{A}(w_1)$ and $w_{2}, w_{4}$ are Catalan words while $w_1$ is a non-Catalan word and $w_3, w_5$ are possibly Catalan words.
\end{proof}

\subsection{System of equations associated to a partition word and a map $S$}\label{sec-equation}
In the situations of Hankel and Toeplitz random matrices in \cite{Hankel}, the associated system of equations plays an important role in computing the moment of the random Hankel and Toeplitz random matrices (namely, the  expectations of the trace of the powers of the corresponding random matrices).   

Following the same idea of  \cite{Hankel},  we introduce the system of equations for our general map $S$ which satisfies the $C$-condition as follows. 

Let $w[j]$ denote the letter in position $j$ of the word $w$. To every partition word $w$ of length $2k$ we associate the following system of equations in unknowns $x_{0},x_{1},\cdots,x_{2k}$:
\begin{equation}\label{equation}
	\begin{aligned}
		S(x_1,x_0) &= S(x_{m_1},x_{m_1 - 1}), \\
		& \quad \text{if } m_1 > 1 \text{ is such that } w[1] = w[m_1], \\
		S(x_2,x_1) &= S(x_{m_2},x_{m_2 - 1}), \\
		& \quad \text{if there is } m_2 > 2 \text{ such that } w[2] = w[m_2], \\
		& \vdotswithin{=} \\
		S(x_i,x_{i-1}) &= S(x_{m_i},x_{m_i - 1}), \\
		& \quad \text{if there is } m_i > i \text{ such that } w[i] = w[m_i], \\
		& \vdotswithin{=} \\
		S(x_{2k-1},x_{2k-2}) &= S(x_{2k},x_{2k-1}), \\
		& \quad \text{if } w[2k - 1] = w[2k],\\
		x_0 &= x_{2k}.
	\end{aligned}
\end{equation}

{\bf \flushleft An explanation of the equation system \eqref{equation}:} One easier way to obtain the equation system \eqref{equation} from a partition word $w$ is given as follows. 
\begin{itemize}
\item First, given a partition word $w$ of length $2k$: 
\[
w = w[1] w[2] \cdots w[2k], 
\]  we write the $(2k+1)$-variables $x_0, x_1, \cdots, x_{2k}$ in the spaces between the letters of the word: 
\[
 ^{x_0} w[1]^{x_1} w[2]^{x_2} \cdots ^{x_{2k-1}} w[2k]^{x_{2k}}. 
\]
 For example, writing the variables between the letters of the word 
 \begin{align}\label{w-exam}
 w = abbacdd e c e
 \end{align}
  we get
\begin{equation}\label{solve-equation}
	^{x_0} a^{x_1} b^{x_2} b^{x_3} a^{x_4} c^{x_5}  d^{x_6}d^{x_7}e^{x_8}c^{x_{9}}e^{x_{10}}
\end{equation}
\item  Then the equations (\ref{equation}) are formed by equating the $S$-value of the variables at the same letter: for each pair $(i, j)$ with $w[i] = w[j]$, we write an equation: 
\[
S(x_{i-1},  x_i)  =   S(x_{j-1}, x_j). 
\]
And finally, we add one more equation for the subsequent calculation of traces:
\[
x_0 = x_{2k}.
\]
\end{itemize}

\begin{example*}
For the partition word $w = abba$, using the above strategy, we write \[
^{x_0} a^{x_1} b^{x_2} b^{x_3} a^{x_4}
\]
And then we get the corresponding equation system:   
\begin{align*}
	S(x_0,x_1) &= S(x_3,x_4),\\
	S(x_1,x_2) &= S(x_2,x_3),\\
	x_0  &=  x_4.
\end{align*}
And for the slightly more complicated partition  word $w$ given in \eqref{w-exam}, by using \eqref{equation}, we obtain the following corresponding equation system:  
\begin{align*}
S(x_0, x_1) & = S(x_3, x_4), \\
S(x_1, x_2) & = S(x_2, x_3),\\
S(x_4, x_5) & = S(x_8, x_9),\\
S(x_5, x_6) & = S(x_6, x_7),\\
S(x_7, x_8) & = S(x_9, x_{10}),\\
x_0 & = x_{10}.
\end{align*}
\end{example*}

{\bf \flushleft The main role of the equation system \eqref{equation}:} We mainly focus on the number of solutions for \eqref{equation}. More precisely, we need to estimate
\[
\#\{(x_0,\cdots,x_{2k}) \in [n]^{2k+1}| x_0, \cdots,x_{2k}\,\, \text{satisfy the equations}\,\, (\ref{equation})\}.
\]

Since there are exactly $k$ distinct letters in a partition word of length $2k$, this is the system of $(k+1)$ equations in $(2k+1)$ unknowns. We solve it for the variables that precede the first occurrence of a letter (we will call them dependent variables), leaving us with $k$ undetermined variables $x_{\alpha_{1}},\ldots, x_{\alpha_{k}} = x_{2k-1}$ that precede the second occurrence of each letter, and with the $(k+1)$-th undetermined variable $x_{2k}$. We also take \eqref{w-exam} as an example, we write the variables between the letters of the word and get \eqref{solve-equation}. Then by definition, $x_0, x_1, x_4, x_5, x_7$ are dependent variables while $x_2, x_3, x_6, x_8, x_9, x_{10}$ are undetermined variables.

Indeed, undetermined variables are free variables. We claim that if $x_{0},\cdots,x_{2k}$ satisfy the $(k+1)$ equations in $(\ref{equation})$, then each $x_i$ has a unique expression in terms of $x_{\alpha_{1}},\cdots, x_{\alpha_{k}}, x_{2k}$:

Indeed, we write the variables between the letters of the word as the form of $(\ref{solve-equation})$ and solve the variables from right to left. Note that $x_{2k-1}, x_{2k}$ are both undetermined variables, we keep moving to left until we meet the first dependent variable, namely $x_{j}$. By our definition, it is a variable that precede the first occurence of some letter, namely $g$. Then there exists another $g$ on the right, hence the variables on both sides of the right $g$ are undetermined variables. Except for this, the variable on the right side of the left $g$ is also an undetermined variable, hence by the coordinatewise injectivity assumption, if there is a solution, then $x_j$ has a unique expression. We will repeat this procedure and the conclusion will be drawn. Hence we can denote 
\[
F_{i}(x_{\alpha_{1}},x_{\alpha_{2}},\cdots,x_{\alpha_{k}},x_{2k})
\]
the unique expression of the $i$-th dependent variable obtained by solving the equations (\ref{equation}), where $i=1,\cdots,k$. If the solution of the $i$-th dependent variable does not exist, we will denote $F_{i}(x_{\alpha_{1}},x_{\alpha_{2}},\cdots,x_{\alpha_{k}},x_{2k}) = -\infty$.

\subsection{Properties for circuits}\label{sec-cir}
For $k, n \in \mathbb{N}$, consider circuits in $\{1, 2, \ldots, n\}$ of length $L(\pi) = k$, that is, mappings $\pi \colon \{0, 1, \ldots, k\} \rightarrow \{1, 2, \ldots, n\}$, such that $\pi(0) = \pi(k)$. 

For a fixed map $S$ which satisfies the $C$-condition, we will say that circuit $\pi$ is $S$-matched, or has self-matched edges, if for every $1 \leq i \leq L(\pi)$ there is $j \neq i$ such that $S(\pi(i - 1), \pi(i)) = S(\pi(j - 1), \pi(j))$.

We will say that a circuit $\pi$ has an edge of order 3, if there are at least three different edges in $\pi$ with the same $S$-value. The following lemma says that generically self-matched circuits have only pair-matches.

\begin{lemma}[{ \cite[Proposition 4.2]{Hankel}}]\label{lem-3ord}
	Fix $r \in \mathbb{N}$. Let $N$ denote the number of $S$-matched circuits in $\{1, \ldots, n\}$ of length $r$ with at least one edge of order 3. Then there is a constant $\widetilde{C}_{r}$ such that
	\[N \leq \widetilde{C}_{r} n^{\lfloor (r+1)/2 \rfloor}.\]
	In particular, we have
	\[\frac{N}{n^{1+r/2}} \xrightarrow{n\to\infty} 0.\]
\end{lemma}

We say that a set of circuits $\pi_1, \pi_2, \pi_3, \pi_4$ is matched if each edge of any one of these circuits is either self-matched, that is, there is another edge of the same circuit with equal $S$-value, or is cross-matched, that is, there is an edge of the other circuit with the same $S$-value (or both). The following bound will be used to prove almost sure convergence of moments.

\begin{lemma}[{\cite[Proposition 4.3]{Hankel}}]\label{quadruples} 
	Fix $r \in \mathbb{N}$. Let $N$ denote the number of matched quadruples of circuits in $\{1, \ldots, n\}$ of length $r$ such that none of them is self-matched. Then there is a constant $\widetilde{C}_{r}$ such that
	\[N \leq \widetilde{C}_{r} n^{2r+2}.\]
\end{lemma}

\section{Proofs of Theorems.}
The proofs of Theorem rely on the well-known relation
$$\int x^{k} d\widehat{\mu}(A)=\frac{1}{n} \Tr A^{k}$$
for an $n \times n$ symmetric matrix A.

\subsection{Reduced to bounded random variable situation}

\begin{proposition}\label{prop-TB}
	If Theorem~\ref{thm-main} holds true for a bounded real random variable $X$ with mean zero and variance 1, then it holds true for a real random variable $X$ with variance 1 and $\mathbb{E}[|X|^{2 + \varepsilon}] < \infty$ for some fixed $\varepsilon > 0$.  
\end{proposition}

In our  proof of Proposition \ref{prop-TB},  we shall need the following lemma  on the sufficient conditions for almost sure convergence in double arrays.
\begin{lemma}[{\cite[Theorem 4, Corollary 3]{Teicher}}]\label{lem-Teicher}
	Let $\{X_n, n=1,2,\cdots\}$ be i.i.d.\,mean-zero random variables and let $\{a_{n, i}, 1 \leq i \leq k_n \uparrow \infty \}$ be a double array of constants. Assume that there exists $0 < p \leq 2$ such that
	\begin{enumerate}
		\item $\mathbb{E}[|X_1|^p] < \infty$,
		\item $\max_{1 \leq i \leq k_n} |a_{n,i}| = O\big( k_n^{-\frac{1}{p}}(\log n)^{-1}\big)$,
		\item $\log k_n = O(\log^2n)$.
	\end{enumerate}
	Then 
	\begin{equation*}
		\sum_{i=1}^{k_n} a_{n,i}X_i \xrightarrow[a.s.]{n\to\infty} 0.
	\end{equation*}
\end{lemma}

The following lemma will be useful for us to complete the proof of Proposition~\ref{prop-TB}:
\begin{lemma}\label{lem-0}
For any $\varepsilon>0$, 
\[
  \lim_{n\to\infty}
  \frac{\max\limits_{1 \leq k \leq n^2} d(k)\,|A_n|^{\frac{2}{2+\varepsilon}}\,\log n}{n^2}
  =0.
\]
\end{lemma}

\begin{proof}
By the estimations~\eqref{d_n} and \eqref{A_n}, there exist constants $M_1>0$ and $M_{2} > 0$ such that for all $n\ge1$,
\begin{equation}\label{An-bound}
  |A_n|\le \frac{M_1\,n^2}{(\log n)^c\,(\log\log n)^{3/2}},
  \quad
\max\limits_{1 \leq k \leq n^2} d(k) \leq M_{2}n^{\frac{\varepsilon}{2+\varepsilon}}.
\end{equation}
Therefore, we can compute that
\begin{align*}
\lim_{n\to\infty}
  \frac{\max\limits_{1 \leq k \leq n^2} d(k)\,|A_n|^{\frac{2}{2+\varepsilon}}\,\log n}{n^2} &\leq \lim_{n\to\infty}\frac{M_{2}n^{\frac{\varepsilon}{2+\varepsilon}}\,\frac{M_{1}^{\frac{2}{2+\varepsilon}} n^{\frac{4}{2+\varepsilon}}}{(\log n)^{\frac{2c}{2+\varepsilon}} (\log \log n)^{\frac{3}{2+\varepsilon}}}\,\log n}{n^2}\\
  &= \lim_{n\to\infty} \frac{M_{1}^{\frac{2}{2+\varepsilon}}M_{2}(\log n)^{1 - \frac{2c}{2+\varepsilon}}}{n^{\frac{\varepsilon}{2+\varepsilon}}(\log \log n)^{\frac{3}{2+\varepsilon}}} = 0
\end{align*}
which completes the proof.
\end{proof}

\begin{proof}[Proof of Proposition \ref{prop-TB}]
	Without loss of generality, we may assume that $\mathbb{E}[X]=0$ in Theorem~\ref{thm-main}. Indeed, from the rank inequality (\ref{rank inequality}) it follows that substracting a rank-1 matrix of the mean $\mathbb{E}[X]$ from matrices $H_{n}$ does not affect the asymptotic distribution of the eigenvalues.
	
	For a fixed $u > 0$, denote
	\[
	m(u) = \mathbb{E}[X_1 \indi_{\{|X_1|>u\}}],
	\] and let \[\sigma^2(u) = \mathbb{E}[X_1^2 \indi_{\{|X_1| \leq u\}}] - m^2(u).\] Clearly, $\sigma^2(u) \leq 1$ and since $\mathbb{E}[X_1] = 0$, $\mathbb{E}[X_1^2] = 1$, we have 
	\begin{align}\label{mu-sigmau}
	\text{$m(u) \rightarrow 0$ and $\sigma(u) \rightarrow 1$ as $u \rightarrow \infty$}.
	\end{align}
	 Let\[\widetilde{X}_1 = X_1 \indi_{\{|X_1|>u\}} - m(u).\] Notice that $\sigma^2(u) = \mathbb{E}[(X_1 - \widetilde{X}_1)^2]$, therefore the bounded random variable\[\widehat{X}_{1} = \frac{X_1 - \widetilde{X}_1}{\sigma(u)}\] has mean zero and variance 1. Denote by $\widehat{H}_{n}$ the corresponding random $n \times n$ Helson matrix constructed from the independent bounded random variables \[\widehat{X}_j := \frac{X_j - \widetilde{X}_j}{\sigma(u)}\] distributed as $\widehat{X}_1$. By the triangle inequality for $d_{BL}(\cdot, \cdot)$ and (\ref{lidskii}),
		\begin{align}\label{dbl-trunc}
		\begin{split}
			d_{BL}^2(\widehat{\mu}(H_n / \sqrt{n}), \widehat{\mu}(\widehat{H}_n / \sqrt{n})) \leq  & 2d_{BL}^2(\widehat{\mu}(H_n / \sqrt{n}), \widehat{\mu}(\sigma(u) \widehat{H}_n / \sqrt{n})) \\
			& \quad + 2d_{BL}^2(\widehat{\mu}(\widehat{H}_n / \sqrt{n}), \widehat{\mu}(\sigma(u) \widehat{H}_n / \sqrt{n})) \\
			\leq&  \frac{2}{n^2} \Tr((H_n - \sigma(u)\widehat{H}_n)^2)\\ &+ \frac{2}{n^2} (1 - \sigma(u))^2 \Tr(\widehat{H}_{n}^2)
			\end{split}
	\end{align}
	and
	\begin{align*}
		\frac{1}{n^2} \Tr((H_n - \sigma(u)\widehat{H}_n)^2) = \frac{1}{n^2} \sum_{1 \leq i,j \leq n} \widetilde{X}_{ij}^{2} = \frac{1}{n^2} \sum_{k \in A_n} m(k;n) \widetilde{X}_{k}^{2}.
	\end{align*}
	By simple calculation, we  get 
\[
	\mathbb{E}[\widetilde{X}_{1}^{2}] = 1 - \sigma^2(u) - 2m(u)^2.
\]
	and hence, by \eqref{mu-sigmau}, 
		\begin{align}\label{X-tilde-0}
	\mathbb{E}[\widetilde{X}_{1}^{2}] \to 0  \quad \text{as $u\to\infty$}.	
	\end{align}
Introduce the mapping
\[
  f: A_{n} \longrightarrow \{1,2,\dots,|A_n|\}
\]
which arranges the elements of \(A_n\) in increasing order and labels them consecutively. Then \(f\) is a bijection and
\[
  \frac{1}{n^2}\sum_{k \in A_{n}} m(k;n)\,\widetilde X_k^2
  \;=\;
  \frac{1}{n^2}\sum_{k = 1}^{|A_{n}|} m\bigl(f^{-1}(k);n\bigr)\,\widetilde X_{f^{-1}(k)}^2.
\]
Denote $Y_{k} = \widetilde X_{f^{-1}(k)}^2$, then we will have
\[
\frac{1}{n^2}\sum_{k \in A_{n}} m(k;n)\,\widetilde X_k^2 =  \frac{1}{n^2}\sum_{k = 1}^{|A_{n}|} m\bigl(f^{-1}(k);n\bigr)\,Y_k
\]
Next we verify the conditions of Lemma~\ref{lem-Teicher}. Take
\[
  k_n = |A_n|,\quad
  a_{n,k} = \frac{m\bigl(f^{-1}(k);n\bigr)}{n^2},\quad
  X_n = Y_n - \mathbb{E}[Y_n],\quad
  p = \frac{2+\varepsilon}{2}.
\]
From \eqref{precise} and \eqref{A_n} we know
\[
  \log|A_n| = O(\log n) = O\bigl((\log n)^2\bigr)
  \quad\text{as }n\to\infty.
\]
Moreover, by Lemma~\ref{lem-0}, as \(n\to\infty\),
\begin{align}\label{upper-bound-m}
\begin{split}
  \max_{1\le k\le |A_n|}\,\frac{m(f^{-1}(k);n)}{n^2}
  &\le
  \max_{1\le k\le |A_n|}\,\frac{d\bigl(f^{-1}(k)\bigr)}{n^2}\\
  &\leq \max_{1\le k\le n^2}\,\frac{d(k)}{n^2}
   =O\!\Bigl(\frac{1}{|A_n|^{2/(2+\varepsilon)}\;\log n}\Bigr).
\end{split}
\end{align}
Hence we have
\[
\frac{1}{n^2} \sum_{k=1}^{|A_n|}m(f^{-1}(k); n) (Y_k- \E[Y_k])  \xrightarrow[n\to\infty]{\mathrm{a.s.}} 0.
\]
Together with 
\[
\frac{1}{n^2}\sum_{k=1}^{|A_n|}m(f^{-1}(k);n) = 1,
\]
we have
\begin{equation}\label{a.s.}
  \frac{1}{n^2}\Tr\!\bigl((H_n-\sigma(u)\,\widehat H_n)^2\bigr)
  \xrightarrow[n\to\infty]{\mathrm{a.s.}}
  \mathbb{E}[\widetilde X_1^2].
\end{equation}
Similarly one shows
\begin{equation}\label{a.s.1}
  \frac{1}{n^2}\Tr\!\bigl(\widehat H_n^2\bigr)
  \xrightarrow[n\to\infty]{\mathrm{a.s.}}
  \mathbb{E}[\widehat X_1^2].
\end{equation}
For large $u$, both $m(u)$ and $1-\sigma(u)$ are arbitrarily small. So, in view of   \eqref{mu-sigmau},   \eqref{dbl-trunc},  \eqref{X-tilde-0}, \eqref{a.s.} and \eqref{a.s.1}, with probability 1 the limiting distance in the bounded Lipschitz metric $d_{BL}$ between $\widehat{\mu}(H_{n} / \sqrt{n})$ and $\widehat{\mu}(\widehat{H}_{n} / \sqrt{n})$ is arbitrarily small, for all $u$ sufficiently large. More precisely, we have 
\begin{align}\label{dbl-to-zero}
\limsup_{u\to\infty}\limsup_{n \to\infty}    d_{BL}(\widehat{\mu}(H_n / \sqrt{n}), \widehat{\mu}(\widehat{H}_n / \sqrt{n})) =0, a.s. 
\end{align}
Thus, if the conclusion of Theorem~\ref{thm-main} holds true for all sequences of independent bounded random variables $\{\widehat{X}_j\}$, with the same limiting distribution $\gamma_{sc}$, then $\widehat{\mu}(H_n / \sqrt{n})$ must have the same weak limit with probability 1.
\end{proof}

\subsection{Moments of the average spectral measure}
\begin{proposition}\label{moments}
	Assume that $S$ is a map satisfying the $C$-condition. Let $X$ be an arbitrary bounded real random variable such that $\mathbb{E}[X]=0$ and $\Var(X)=1$. Let $H_{n}^{S}$ be the corresponding random $n \times n$ $S$-patterned matrix. Then for $k \in \mathbb{N}$,
	\begin{equation}\label{even}
		\lim_{n \to \infty} \frac{1}{n^{k+1}} \mathbb{E}\left[\Tr((H_n^{S})^{2k})\right] = \frac{(2k)!}{(k+1)!k!},
	\end{equation}
	and
	\begin{equation}\label{odd}
		\lim_{n \to \infty} \frac{1}{n^{k+1/2}} \mathbb{E}\left[\Tr((H_n^{S})^{2k-1})\right] = 0.
	\end{equation}
\end{proposition}

\begin{proof}
	For a circuit $\pi \colon \{0, 1, \ldots, r\} \to \{1, 2, \ldots, n\}$ write
	\begin{equation}\label{X_{pi}}
		X_{\pi} = \prod_{i=1}^{r} X_{S(\pi(i-1),\pi(i))}.
	\end{equation}
	Then,
	\begin{align}\label{trace}
		\begin{split}
			\mathbb{E}[\Tr((H_n^{S})^{r})] 
			&= \sum_{i_1,\cdots,i_{r}=1}^{n} \mathbb{E}[H_{n}^{S}(i_1,i_2)H_{n}^{S}(i_2,i_3) \cdots H_{n}^{S}(i_r,i_1)]\\
			&=\sum_{i_1,\cdots,i_{r}=1}^{n} \mathbb{E}[X_{S(i_1,i_2)} X_{S(i_2,i_3)} \cdots X_{S(i_r,i_1)}]\\ 
			&= \quad \sum_{\pi} \mathbb{E}[X_{S(\pi(0),\pi(1))}X_{S(\pi(1),\pi(2))} \cdots X_{S(\pi(r-1),\pi(r))}]\\ 
			&= \quad \sum_{\pi} \mathbb{E}[X_{\pi}].
		\end{split}
	\end{align}
	where the sum is over all circuits in $\{1, \ldots, n\}$ of length $r$, and by H\"older's inequality, for any finite set $\Pi$ of circuits of length $r$
	\begin{equation}\label{holder}
		\Big| \sum_{\pi \in \Pi} \mathbb{E}[X_{\pi}] \Big| \leq \mathbb{E}[|X|^r]\# \Pi.
	\end{equation}
	With the random variables $X_i \, (i \in I)$ independent and of mean zero, the term $\mathbb{E}[X_{\pi}]$ vanishes for every circuit $\pi$ with at least one unpaired $X_i$. By (\ref{X_{pi}}), in the current setting paired variables correspond to an $S$-matching in the circuit $\pi$. Hence, only $S$-matched circuits can make a nonzero contribution to (\ref{trace}). 
	
	If $r = 2k - 1 > 0$ is odd, then each $S$-matched circuit $\pi$ of length $r$ must have an edge of order 3. From (\ref{holder}) and Lemma~\ref{lem-3ord} we get $\left| \mathbb{E}[\Tr((H_n^{S})^{2k-1})] \right| \leq C n^{k}$, proving (\ref{odd}).
	
	When $r = 2k$ is an even number, let $\Pi^{S}(k,n)$ be the set of all $S$-matched circuits $\pi \colon \{0, 1, \ldots, 2k\} \rightarrow \{1, \ldots, n\}$ consisting of $k$ distinct pairs. Recall that $\mathbb{E}X_{\pi} = 1$ for any $\pi \in \Pi^{S}(k,n)$ [see (\ref{X_{pi}})]. Further, with any $S$-matched circuit not in $\Pi^{S}(k,n)$ having an edge of order 3, it follows from (\ref{holder}) and Lemma~\ref{lem-3ord} that
	\[
	\lim_{n \to \infty} \frac{1}{n^{k+1}} \left| \mathbb{E}[\Tr((H_n^{S})^r)] - \# \Pi^{S}(k,n) \right| = 0.
	\]
	Label the circuits in $\Pi^{S}(k,n)$ by the partition words $w$ which list the positions of the pairs of $S$-matches along $\{1, \ldots, 2k\}$, with the corresponding partition $\Pi^{S}(k,n) = \bigsqcup_{w} \Pi^{S}(k,n,w)$ into equivalence classes. 
	Recall that in \S\ref{sec-equation}
	\[
	F_{i}(x_{\alpha_{1}},x_{\alpha_{2}},\cdots,x_{\alpha_{k}},x_{2k})
	\]
	is the unique expression of the $i$-th dependent variable obtained by solving the equations (\ref{equation}), and if the solution of the $i$-th dependent variable does not exist, we will denote $F_{i}(x_{\alpha_{1}},x_{\alpha_{2}},\cdots,x_{\alpha_{k}},x_{2k}) = -\infty$.
	To every such partition word $w$, denote
	\begin{align*}
		& p_{n}^{S}(w) 
		= \frac{\#\{(x_0,\cdots,x_{2k}) \in [n]^{2k+1}| x_0, \cdots,x_{2k}\,\, \text{satisfy the equations}\,\, (\ref{equation})\}}{n^{k+1}}\\
		&= \frac{\#\{(x_{\alpha_{1}},x_{\alpha_{2}},\cdots,x_{\alpha_{k}},x_{2k}) \in [n]^{k+1}|F_{i}(x_{\alpha_{1}},x_{\alpha_{2}},\cdots,x_{\alpha_{k}},x_{2k}) \in [n], i=1,2,\cdots,k\}}{n^{k+1}}.
	\end{align*}
	Note that the numerator of $p_{n}^{s}(w)$ is the number of $S$-matched circuits in $\{1,\cdots,n\}$ of length $r$ and $\#\Pi^{S}(k,n,w)$ is the number of all not only $S$-matched but also pair matched circuits in $\{1,\cdots,n\}$ of length $r$. Hence by Lemma~\ref{lem-3ord},
	\begin{equation}\label{pns}
		\lim_{n \to \infty} p_n^{S}(w) = \lim_{n \to \infty} \frac{\#\Pi^{S}(k,n,w)}{n^{k+1}}.
	\end{equation}
	The following lemma will complete the proof of (\ref{even}), and with it, that of proposition~\ref{moments}.
\end{proof}

\begin{lemma}\label{key}
	Denote $\Pi^{S}(k,n)$ the set of all $S$-matched circuits $\pi \colon \{0, 1, \ldots, 2k\} \rightarrow \{1, \ldots, n\}$ consisting of $k$ distinct pairs. Then the cardinality of the set $\Pi^{S}(k,n)$ satisfies
	\[
	\lim_{n \to \infty} \frac{1}{n^{k+1}} \# \Pi^{S}(k,n) = \frac{(2k)!}{(k+1)!k!}
	.\]
\end{lemma}

Lemma \ref{key} will be proved in  \S \ref{sec-key-lem}. 

\subsection{Some key lemmas}\label{sec-key-lem}
We will introduce the key lemmas in our article which will play an important role in the proof of Lemma~\ref{key}.
\begin{lemma}\label{p_n(w)}
	Given a partition word $w$, denote $w'$ the reduced word obtained from $w$. Then $p_n^{S}(w) = p_n^{S}(w')$, Specially, $p_n^{S}(\emptyset) = 1$.
\end{lemma}

\begin{proof}
	It suffices to prove
	\[
	p_n^{S}(w_{1}aaw_{2}) = p_n^{S}(w_{1}w_{2}),
	\]
	where $w =w_{1}aaw_{2}$ is a partition word of length $|w| = 2k$. Denote  the letter in $w_1$ which is next to letter $a$ as $b$ while the letter in $w_2$ which is next to letter $a$ as $c$. Writing the variables between the letters of the word, suppose it has the form:
	\[
	\cdots ^{x_{\alpha}}b^{x_{\alpha + 1}}a^{x_{\alpha + 2}}a^{x_{\alpha + 3}}c^{x_{\alpha + 4} \cdots}.
	\]
	Consider the equation of $w_{1}aaw_{2}$ determined by (\ref{equation}) which involves $x_{\alpha+1}, x_{\alpha+2}, x_{\alpha+3}$:
	\begin{equation*}
		S(x_{\alpha+1}, x_{\alpha+2}) = S(x_{\alpha+2}, x_{\alpha+3}),
	\end{equation*}
	which implies $x_{\alpha+1} = x_{\alpha+3}$ by the coordinatewise injectivity of the map $S$. Suppose that 
	\[
	p_n^{S}(w_{1}aaw_{2}) = \frac{M_1}{n^{k+1}}
	\]
	and
	\[
	p_n^{S}(w_{1}w_{2}) = \frac{M_2}{n^k}.
	\]
	Note that $x_{\alpha + 2}$ is an undetermined variable, hence $x_{\alpha+1} = x_{\alpha+3}$ implies $M_1 = nM_2$. Therefore,
	\[
	p_n^{S}(w_{1}aaw_{2}) = p_n^{S}(w_{1}w_{2}).
	\] 
\end{proof}

\begin{lemma}\label{noncartalan p_n(w)}
	For any non-Catalan partition word $w$, we have 
	\[
	\lim_{n \to \infty} p_n^{S}(w) = 0.
	\]
\end{lemma}

\begin{proof}
	Denote $w'$ the reduced word obtained from $w$, suppose $|w'| = 2k$, then according to Lemma~\ref{p_n(w)}, we have $p_n^{S}(w) = p_n^{S}(w')$. By Lemma~\ref{deleting}, $w'$ must be of one of the following forms with writing the variables between the letters of the word:
	
	Case 1:
	\[
	\cdots ^{x_{\alpha}}a^{x_{\alpha+1}}b^{x_{\alpha+2}} \cdots ^{x_{\beta}}a^{x_{2k}},
	\]
	where $b$ is the second occurrence here, hence $x_{\alpha+1}, x_{\beta}, x_{2k}$ are undetermined variables while $x_{\alpha}$ is a dependent variable. Hence
		\begin{align}\label{sdc-1}
            \begin{split}
			p_{n}^{S}(w') &= \frac{\#\{(x_{\alpha_{1}},x_{\alpha_{2}},\cdots,x_{\alpha_{k}},x_{2k}) \in [n]^{k+1}|F_{i}(x_{\alpha_{1}},x_{\alpha_{2}},\cdots,x_{\alpha_{k}},x_{2k}) \in [n], i=1,2,\cdots,k\}}{n^{k+1}}\\
			&\leq\frac{\#\{(x_{\alpha},x_{\alpha_{1}},x_{\alpha_{2}},\cdots,x_{\alpha_{k}},x_{2k}) \in [n]^{k+2}|S(x_{\alpha},x_{\alpha + 1}) = S(x_{\beta},x_{2k})\}}{n^{k+1}}\\
			&= \frac{\#\{(x_{\alpha},x_{\alpha + 1},x_{\beta},x_{2k}) \in [n]^4| S(x_{\alpha},x_{\alpha + 1}) = S(x_{\beta},x_{2k})\}}{n^3}.
            \end{split}
	\end{align}
	Then by the small dimension assumption on the map $S$,
	\[
	\lim_{n \to \infty} p_n^{S}(w) = 0.
	\]
	
	Case 2:
	\[
	\cdots ^{x_{\alpha}}a^{x_{\alpha+1}}b^{x_{\alpha+2}} \cdots ^{x_{\beta}}a^{x_{\beta+1}}c \cdots,
	\]
	where $b,c$ are the second occurrences here, hence $x_{\alpha+1}, x_{\beta}, x_{\beta+1}$ are undetermined variables while $x_{\alpha}$ is a dependent variable. Hence
		\begin{align}\label{sdc-2}
            \begin{split}
			p_{n}^{S}(w') &= \frac{\#\{(x_{\alpha_{1}},x_{\alpha_{2}},\cdots,x_{\alpha_{k}},x_{2k}) \in [n]^{k+1}|F_{i}(x_{\alpha_{1}},x_{\alpha_{2}},\cdots,x_{\alpha_{k}},x_{2k}) \in [n], i=1,2,\cdots,k\}}{n^{k+1}}\\ 
			&\leq\frac{\#\{(x_{\alpha},x_{\alpha_{1}},x_{\alpha_{2}},\cdots,x_{\alpha_{k}},x_{2k}) \in [n]^{k+2}|S(x_{\alpha},x_{\alpha + 1}) = S(x_{\beta},x_{\beta + 1})\}}{n^{k+1}}\\
			&= \frac{\#\{(x_{\alpha},x_{\alpha + 1},x_{\beta},x_{\beta + 1}) \in [n]^4| S(x_{\alpha},x_{\alpha + 1}) = S(x_{\beta},x_{\beta + 1})\}}{n^3}.
            \end{split}
	\end{align}
	Then by the small dimension assumption,
	\[
	\lim_{n \to \infty} p_n^{S}(w) = 0.
	\]
\end{proof}

\begin{proof}[Proof of Lemma~\ref{key}]
	By $(\ref{pns})$, it suffices to show that
	\begin{equation*}
		\lim_{n \to \infty} \sum_{|w| = 2k} p_n^{S}(w) = \frac{(2k)!}{(k+1)!k!}.
	\end{equation*}
	There are two cases:
	
	Case 1:
	$w$ is a Catalan word, then by Lemma~\ref{p_n(w)},  
	\[
	p_n^{S}(w) = p_n^{S}(\emptyset)=1,
	\] 
	
	Case 2:
	$w$ is a non-Catalan word, by Lemma~\ref{noncartalan p_n(w)}, 
	\[
	\lim_{n \to \infty} p_n^{S}(w) = 0.
	\]
	Therefore, by Lemma~\ref{Catalan}
	\begin{equation*}
		\lim_{n \to \infty} \sum_{|w| = 2k} p_n^{S}(w) = C_k = \frac{(2k)!}{(k+1)!k!}.
	\end{equation*}
\end{proof}

\subsection{Concentration of moments of the spectral measure}
\begin{proposition}\label{4-order trace}
	Assume that $S$ is a map satisfying the $C$-condition. Let $X$ be an arbitrary bounded real random variable such that $\mathbb{E}[X]=0$ and $\Var(X)=1$. Let $H_{n}^{S}$ be the corresponding random $n \times n$ $S$-patterned matrix. Then fix $r \in \mathbb{N}$, there is $\widetilde{C}_{r} < \infty$ such that for all $n \in \mathbb{N}$ we have
	\[
	\mathbb{E}\left[\left(\Tr((H_n^{S})^r) - \mathbb{E}\Tr((H_n^{S})^r)\right)^4\right] \leq \widetilde{C}_{r} n^{2r+2}.
	\]
\end{proposition}

\begin{proof}
	Using the circuit notation of (\ref{X_{pi}}) we have that
	\begin{equation}\label{4-order moments}
		\mathbb{E}\left[\left(\Tr((H_n^{S})^r) - \mathbb{E}\Tr((H_n^{S})^r)\right)^4\right] = \sum_{\pi_1,\pi_2,\pi_3,\pi_4} \mathbb{E}\left[\prod_{j=1}^{4} \left(X_{\pi_j} - \mathbb{E}[X_{\pi_j}]\right)\right],
	\end{equation}
	where the sum is taken over all circuits $\pi_j$, $j = 1, \ldots, 4$ on $\{1, \ldots, n\}$ of length $r$ each. With the random variables $X_i\,(i \in I)$ independent and of mean zero, any circuit $\pi_k$ which is neither self-matched nor cross-matched has $\mathbb{E}[X_{\pi_k}] = 0$ and
	\[
	\mathbb{E}\left[\prod_{j=1}^{4} \left(X_{\pi_j} - \mathbb{E}[X_{\pi_j}]\right)\right] = \mathbb{E}\left[X_{\pi_k} \prod_{j\in \{1,2,3,4\}\setminus\{k\}}  \left(X_{\pi_j} - \mathbb{E}[X_{\pi_j}]\right)\right] = 0.
	\]
	Further, if one of the circuits, say $\pi_1$, is only self-matched, that is, has no cross-matched edge, then obviously
	\[
	\mathbb{E}\left[\prod_{j=1}^{4} \left(X_{\pi_j} - \mathbb{E}[X_{\pi_j}]\right)\right] = \mathbb{E}\left[X_{\pi_1} - \mathbb{E}[X_{\pi_1}]\right]\mathbb{E}\left[\prod_{j=2}^{4} \left(X_{\pi_j} - \mathbb{E}[X_{\pi_j}]\right)\right] = 0.
	\]
	Therefore, it suffices to take the sum in (\ref{4-order moments}) over all $S$-matched quadruples of circuits on $\{1, \ldots, n\}$, such that none of them is self-matched. By Lemma~\ref{quadruples}, there are at most $\widetilde{C}_{r} n^{2r+2}$ such quadruples of circuits, and with $|X|$ (hence $|X_{\pi}|$) bounded, this completes the proof.
\end{proof}

\subsection{Proofs of Helson and general cases}
\begin{proof}[Proof of Theorem \ref{thm-gen-semi}]
	The rank inequality ~\eqref{rank inequality} implies that without loss of generality we may assume that the random variable $X$ is centered. By Proposition~\ref{moments} the odd moments of the average measure $\mathbb{E}[\widehat{\mu}(H_n^{S} / \sqrt{n})]$ converge to 0, and the even moments converge to $C_k$. Recall that $\gamma_{sc}$ is a probability measure and characterized by its moments. By Chebyshev's inequality we have from Proposition~\ref{4-order trace} that for any $\delta > 0$ and $k, n \in \mathbb{N}$,
	\[
	\mathbb{P}\left[\left| \int x^k \,d\widehat{\mu}\left(\frac{H_n^{S}}{\sqrt{n}}\right) - \int x^k \,d\mathbb{E}[\widehat{\mu}\left(\frac{H_n^{S}}{\sqrt{n}}\right)] \right| > \delta \right] \leq \widetilde{C}_{k} \delta^{-4} n^{-2}.
	\]
	Thus, by the Borel--Cantelli lemma,
	\[
	\int x^k \,d\widehat{\mu}\left(\frac{H_n^{S}}{\sqrt{n}}\right) \xrightarrow[a.s.]{n\to\infty} \int x^k \,d\gamma_{sc}
	\]
	for every $k \in \mathbb{N}$. Since the moments determine $\gamma_{sc}$ uniquely, by Lemma~\ref{m to w} we have the weak convergence of $\widehat{\mu}(H_n^{S} / \sqrt{n})$ to $\gamma_{sc}$.
\end{proof}

\begin{proposition}\label{multi}
	We have   
	\[
	\lim_{n \to \infty} \frac{\# \bigl\{(x,y,z) \in [n]^3| w=\frac{xy}{z} \in [n]\bigr\}}{n^3} = 0.
	\]
\end{proposition}

\begin{proof}
	Recall that $d(k) = \#D(k)$ is the number of the divisors of $k$ and 
	\[
        m(k;n) =   \#\{(i,j) | i,j \in [n],ij=k\}
        \]
is the number of times $k$ appears in the $n \times n$ multiplication table. Hence we have
	\begin{align*}
		\# \bigl\{(x,y,z) \in [n]^3| w = \frac{xy}{z} \in [n]\bigr\} &= \sum_{z=1}^{n} \# \bigl\{(x,y) \in [n]^2| xy \in \{z,2z,\cdots,nz\} \bigr\}\\
		&= \sum_{z=1}^{n} \sum_{k=1}^{n} m(kz;n).
	\end{align*}
	By $(\ref{d_n})$, for any $0< \varepsilon < \frac{1}{2}$, as $n\to\infty$, we have
	\begin{align*}
		\sum_{z=1}^{n} \sum_{k=1}^{n} m(kz;n) &\leq \sum_{z=1}^{n} \sum_{k=1}^{n} d(kz)\\ 
		&\leq \sum_{z=1}^{n} \sum_{k=1}^{n} o((kz)^\varepsilon)\\
		&\leq \sum_{z=1}^{n} \sum_{k=1}^{n} o(n^{2\varepsilon})\\
		&= o(n^2 \cdot n^{2 \epsilon}) = o(n^{2+ 2\epsilon}).
	\end{align*}
	Hence
	\[
	\lim_{n \to \infty} \frac{\#\{(x,y,z) \in [n]^3| w=\frac{xy}{z} \in [n]\}}{n^3} = 0.
	\]
\end{proof}

\begin{proof}[Proof of Theorem~\ref{thm-main}]
	By Proposition~\ref{prop-TB}, it suffices to show that $S(x,y) = xy$ satisfies the $C$-condition. Clearly, the symmetry and coordinatewise injectivity are satisfied.
	
	Now we turn to verify the small dimension condition. Note that
	\[
	\frac{\#\bigl\{(x,y,z,w) \in [n]^4| xy = zw \bigr\}}{n^3}= \frac{\#\bigl\{(x,y,z) \in [n]^3| w=\frac{xy}{z} \in [n]\bigr\}}{n^3}.
	\]
	Hence by Proposition~\ref{multi},
	\[
	\lim_{n \to \infty} \frac{\#\bigl\{(x,y,z,w) \in [n]^4| xy = zw \bigr\}}{n^3} = 0.
	\]
	Therefore,  $S(x,y) = xy$ satisfies the $C$-condition.
\end{proof}

\section{Examples for some specific maps}
In this section, we apply Theorem~\ref{thm-gen-semi} to the maps $S_{\alpha}(x,y) = x^2 + y^2 + \alpha xy \, (\alpha \geq 0)$. Recall that in \S\ref{introduction}, for $n \in \mathbb{N}$, the corresponding random $n \times n$ $S_{\alpha}$-patterned matrix $H_{n}^{S_{\alpha}}$ is defined by
\[
H_{n}^{S_\alpha}=[X_{i^2+j^2+\alpha ij}]_{1 \leq i,j \leq n}.
\] 
If $\alpha = 2$, then $S_{2}(x,y)=(x+y)^2$. Hence the corresponding $H_{n}^{S_2}$ is just the usual random $n \times n$ Hankel matrix studied in \cite{Hankel}. Therefore, we suppose $\alpha \neq 2$ below. 

\begin{lemma}\label{multi to general}
	Assume that $\alpha > 0, \alpha \neq 2$. Then $S_{\alpha}(x,y)$ satisfies the $C$-condition.
\end{lemma}

\begin{proof}
	Clearly, $S_{\alpha}(x,y)$ is symmetric and it is simple to verify that $S_{\alpha}(x,y)$ is coordinatewise increasing and hence coordinatewise injective. 
	
	Now we turn to verify the small dimension condition \eqref{sda} for $S_{\alpha}(x,y)$. We need to estimate the cardinality of the set:
    \begin{equation*}
        \mathcal{S}_{n}^{\alpha} := \left\{(x,y,z,w) \in [n]^4 : S_{\alpha}(x,y) = S_{\alpha}(z,w)\right\}.
    \end{equation*} 
    Note that 
	\begin{align*}
		\begin{split}
			S_{\alpha}(x,y) = x^2 + y^2 + \alpha xy &= x^2 + y^2 + 2xy + (\alpha - 2)xy\\
			&= (x+y)^2 + \frac{\alpha - 2}{4} [(x+y)^2 - (x-y)^2]\\
			&=\frac{2 + \alpha}{4}(x+y)^2 + \frac{2-\alpha}{4}(x-y)^2.
		\end{split}
	\end{align*}
	  Hence the equation $S_{\alpha}(x,y)=S_{\alpha}(z,w)$ is equivalent to
	\begin{equation}\label{square}
	(2+\alpha)(x+y)^2 + (2-\alpha)(x-y)^2 = (2+\alpha)(z+w)^2 + (2-\alpha)(z-w)^2.
	\end{equation}
    Denote 
	$x+y=a,|x-y|=b, z+w=c, |z-w|=d$. Since $x,y,z,w \in [n]$, we have $a,c \in [2,2n] \cap \mathbb{Z}$ and $ b,d \in [0,n-1] \cap \mathbb{Z}$. By (\ref{square}), the equation $S_{\alpha}(x,y)=S_{\alpha}(z,w)$ is equivalent to
	\[
	(2+\alpha)a^2 + (2-\alpha)b^2 = (2+\alpha)c^2 + (2-\alpha)d^2,
	\]
	i.e.
	\begin{equation}\label{simplify}
		(2+\alpha)(a+c)(a-c) = (2-\alpha)(d+b)(d-b).
	\end{equation}
	Denote $a+c=p,|a-c|=q,d+b=s,|d-b|=t$, then $p \in [4,4n] \cap \mathbb{Z}$, $q,s \in [0,2n-2] \cap \mathbb{Z}$ and $t \in [0,n-1] \cap \mathbb{Z}$. For convenience, we introduce the following set
    \begin{equation*}
        \mathcal{H}_{n}^{\alpha} := \left\{(p,q,s,t) \in \mathbb{Z}_{\geq 0}^{4}: 4 \leq p \leq 4n, 0 \leq q,s \leq 2n-2, 0 \leq t \leq n-1, (2+\alpha)pq = |2-\alpha|st \right\}.
    \end{equation*}
    Then the above process defines a map
    \begin{equation*}
        g: \mathcal{S}_{n}^{\alpha} \rightarrow \mathcal{H}_{n}^{\alpha}.
    \end{equation*}
    Moreover, for any $(p,q,s,t) \in \mathcal{H}_{n}^{\alpha}$, we claim that
    \begin{equation}\label{ori-image}
        \# g^{-1}(p,q,s,t) \leq 16.
    \end{equation}
    In fact, for any $(p,q,s,t)\in\mathcal{H}_{n}^{\alpha}$, we first reconstruct $(a,c)$ from $(p,q)$. From
    \[
    p=a+c,\qquad q=|a-c|
    \]
    we see that either $a\ge c$ and
    \[
    (a,c)=\Bigl(\frac{p+q}{2},\frac{p-q}{2}\Bigr),
    \]
    or $c\ge a$ and
    \[
    (a,c)=\Bigl(\frac{p-q}{2},\frac{p+q}{2}\Bigr).
    \]
    Thus for fixed $(p,q)$ there are at most $2$ possible pairs $(a,c)$
    (ignoring the additional restrictions that $a,c\in[2,2n]\cap\mathbb{Z}$,
    which can only reduce the number of possibilities). 
    
    Similarly, from a given $(s,t)$, using pair
    \[
    s=d+b,\qquad t=|d-b|
    \]
    we obtain at most $2$ possible pairs $(b,d)$:
    \[
    (b,d)=\Bigl(\frac{s+t}{2},\frac{s-t}{2}\Bigr)
    \quad\text{or}\quad
    (b,d)=\Bigl(\frac{s-t}{2},\frac{s+t}{2}\Bigr).
    \]
    Hence for fixed $(p,q,s,t)$ there are at most $4$ possible quadruples $(a,b,c,d)$.
    
    Next, we reconstruct $(x,y)$ from $(a,b)$. Since
    \[
    a=x+y,\qquad b=|x-y|,
    \]
    either $x\ge y$ and
    \[
    (x,y)=\Bigl(\frac{a+b}{2},\frac{a-b}{2}\Bigr),
    \]
    or $y\ge x$ and
    \[
    (x,y)=\Bigl(\frac{a-b}{2},\frac{a+b}{2}\Bigr).
    \]
    Thus for each admissible pair $(a,b)$ there are at most $2$ possible pairs
    $(x,y)$ (again, the conditions $x,y\in[n]$ can only decrease this number).
    The same reasoning applies to $(c,d)$ and $(z,w)$: for each $(c,d)$ there are
    at most $2$ possible pairs $(z,w)$. Therefore, for each fixed $(a,b,c,d)$ there are at most $4$ possible quadruples $(x,y,z,w)$. Combining this with the $4$ possibilities for
    $(a,b,c,d)$, we obtain
    \[
    \#\,g^{-1}(p,q,s,t)\le 4\times 4 = 16,
    \]
    which proves \eqref{ori-image}.
    
    Therefore, by \eqref{ori-image} we can get 
    \begin{equation*}
        \# \mathcal{S}_{n}^{\alpha} \leq 16 \cdot \# \mathcal{H}_{n}^{\alpha} \leq 16 \cdot \# \mathcal{T}_{n}^{\alpha}
    \end{equation*}
    where
    \begin{equation*}
      \mathcal{T}_{n}^{\alpha}:=  \left\{(p,q,s,t) \in \mathbb{Z}_{\geq 0}^{4}: 1 \leq p \leq 4n, 0 \leq q,s \leq 4n, 0 \leq t \leq 4n, (2+\alpha)pq = |2-\alpha|st \right\}.
    \end{equation*}
    Note that 
    \begin{equation*}
        \mathcal{T}_{n}^{\alpha} = \bigl\{(p,q,s,t) \in [4n]^4| (2+\alpha)pq = |2-\alpha| st\bigr\} \bigsqcup \mathcal{Q}_{n}^{\alpha}
    \end{equation*}
    where
    \begin{equation*}
    \mathcal{Q}_{n}^{\alpha} := \left\{(p,q,s,t)\in\{0,1,\ldots,4n\}^4: 1\le p\le 4n,\ q=0,\ s \cdot t=0\right\}.
    \end{equation*}
	We can compute that $\# \mathcal{Q}_{n}^{\alpha} = 4n(8n+1)$. Therefore, for any $0 < \varepsilon < \frac{1}{2}$, we can get
	\begin{align*}
		\lim_{n \to \infty} \frac{\# \mathcal{S}_{n}^{\alpha}}{n^3} &\leq \lim_{n \to \infty} \frac{16 \cdot \# \mathcal{T}_{n}^{\alpha}}{n^3} \\
        &= \lim_{n \to \infty} \frac{16 \cdot \#\bigl\{(p,q,s,t) \in [4n]^4| pq = \frac{|2-\alpha|}{2+\alpha}st\bigr\} + 64n(8n+1)}{n^3}\\ 
		&=16 \cdot \lim_{n \to \infty} \frac{\sum_{s=1}^{4n}\sum_{t=1}^{4n}m(\frac{|2-\alpha|}{2+\alpha}st;4n)}{n^3}\\ 
		&\leq 16 \cdot \lim_{n \to \infty} \frac{\sum_{s=1}^{4n}\sum_{t=1}^{4n}d(\frac{|2-\alpha|}{2+\alpha}st)}{n^3}\\ 
		&= \lim_{n \to \infty} \frac{O(n^{2} \cdot n^{2\epsilon})}{n^3} = 0.
	\end{align*}
	Hence the small dimension condition \eqref{sda} for $S_{\alpha}(x,y)$ is satisfied.
\end{proof}

Now we will explain why the small dimension condition does not hold for $\alpha = 2$. 
\begin{lemma}
   We have 
   \[
   \lim_{n \to \infty} \frac{\#\bigl\{(x,y,z,w) \in [n]^4| S_{2}(x,y) = S_{2}(z,w)\bigr\}}{n^3} = \frac{2}{3}.
   \]
\end{lemma}

\begin{proof}
Note that
\[
\#\bigl\{(x,y,z,w) \in [n]^4| S_{2}(x,y) = S_{2}(z,w)\bigr\} = \#\bigl\{(x,y,z) \in [n]^3| x+y-z \in [n]\bigr\}.
\]
Denote $s = x + y$ and then $s$ ranges over the integers $2,3,\dots,2n$. For each fixed $s$, the number of pairs $(x,y)\in[n]^2$ with $x+y=s$ is
\[
a_s \;=\;\#\bigl\{(x,y)\in[n]^2: x+y=s \bigr\}
=\begin{cases}
s-1, &2\le s\le n+1,\\
2n+1-s, &n+2\le s\le 2n.
\end{cases}
\]
Next, for a given $s$, we require
\[
1 \,\le\, x+y-z = s-z \,\le\, n
\quad\Longleftrightarrow\quad
s-n \le z \le s-1,
\]
with $z\in [n]$.  Hence the number of valid $z$–values is
\[
b_s \;=\;\#\bigl\{\,z\in[n]:s-n\le z\le s-1\bigr\}
=\begin{cases}
s-1, &2\le s\le n+1,\\
2n+1-s, &n+2\le s\le 2n,
\end{cases}
\]
so in fact $b_s=a_s$. Therefore
\[
\#\bigl\{(x,y,z) \in [n]^3| x+y-z \in [n]\bigr\}
=\sum_{s=2}^{2n}a_s\,b_s
=\sum_{s=2}^{2n}a_s^2
=\sum_{k=1}^{n}k^2 \;+\;\sum_{k=1}^{n-1}k^2
=n^2 \;+\;2\sum_{k=1}^{n-1}k^2.
\]
As we all know that
\[
\sum_{k=1}^{n-1}k^2
=\frac{(n-1)n(2n-1)}{6}
=\frac{2n^3-3n^2+n}{6},
\]
so
\[
\#\bigl\{(x,y,z) \in [n]^3| x+y-z \in [n]\bigr\}
=n^2 \;+\;2\cdot\frac{2n^3-3n^2+n}{6}
=\frac{2n^3 + n}{3}
=\frac23\,n^3 + \frac13\,n.
\]
Hence
\[
\lim_{n \to \infty} \frac{\#\bigl\{(x,y,z,w) \in [n]^4| S_{2}(x,y) = S_{2}(z,w)\bigr\}}{n^3} = \frac{2}{3},
\]
which means that the small dimension condition does not hold for $\alpha = 2$. 
\end{proof}

Now let's give the result on the weak limit of $\widehat{\mu}(H_{n}^{S_\alpha} / \sqrt{n})$:
\begin{theorem}
	Assume that $\alpha \geq 0, \alpha \neq 2$. Let $X$ be an arbitrary bounded real random variable such that  $\Var(X)=1$. Let $H_{n}^{S_\alpha}$ be the corresponding random $n \times n$ $S_{\alpha}$-patterned matrix. Then, with probability 1, $\widehat{\mu}(H_{n}^{S_\alpha} / \sqrt{n})$ converges weakly as $n \rightarrow \infty$ to the Wigner semi-circular law $\gamma_{sc}$.
\end{theorem}

\begin{proof}
	It follows immediately from Theorem~\ref{thm-gen-semi} and Lemma~\ref{multi to general}.
\end{proof}


\end{document}